\documentclass[twoside,a4paper,reqno,11pt]{amsart}
\usepackage{amsfonts, amsbsy, amsmath, amsthm, amssymb, latexsym}
\usepackage{mathrsfs}
\usepackage{pdfsync}
\usepackage[top=30mm,right=30mm,bottom=30mm,left=30mm]{geometry}
\usepackage{tikz}
\usetikzlibrary{matrix,arrows}
\usepackage{array}

\usepackage{stmaryrd}
\usepackage{bm}

\makeatletter
\newcommand{\imod}[1]{\allowbreak\mkern4mu({\operator@font mod}\,\,#1)}
\makeatother

\newcommand{\G}{\Gamma}

\newcommand{\la}{\langle}
\newcommand{\ra}{\rangle}

\newcommand{\ds}{\displaystyle} \newcommand{\vs}{\vspace{3mm}}

\newcommand{\ol}{\overline}

\DeclareMathOperator*{\lcm}{lcm}
\DeclareMathOperator*{\Soc}{Soc}
\DeclareMathOperator*{\diag}{diag}
\DeclareMathOperator*{\End}{End}

\newtheorem{problem}{Problem}

\newtheorem{thm}{Theorem}[section]
\newtheorem{prop}[thm]{Proposition}
\newtheorem{lem}[thm]{Lemma}
\newtheorem{cor}[thm]{Corollary}

\newtheorem{prob}[thm]{Problem}

\theoremstyle{definition}
\newtheorem{defn}[thm]{Definition}
\newtheorem{remk}[thm]{Remark}

\newtheorem*{theoremA}{Theorem A}
\newtheorem*{theoremB}{Theorem B}
\newtheorem*{theoremC}{Theorem C}

\begin{document}

 \author{Elisa Covato}
 \address{School of Mathematics, University of Bristol, Bristol BS8 1TW, UK}
 \email{elisa.covato@bristol.ac.uk}

\title[On boundedly generated subgroups of profinite groups]{On boundedly generated subgroups \\ of profinite groups}

\begin{abstract}
In this paper we investigate the following general problem. Let $G$ be a group and let $i(G)$ be a property of $G$. Is there an integer $d$ such that $G$ contains a $d$-generated subgroup $H$ with $i(H)=i(G)$? Here we consider the case where $G$ is a profinite group and $H$ is a closed subgroup, extending earlier work of Lucchini and others on finite groups. For example, we prove that $d=3$ if $i(G)$ is the prime graph of $G$, which is best possible, and we show that $d=2$ if $i(G)$ is the exponent of a finitely generated prosupersolvable group $G$.
\end{abstract}

\date{\today}


\maketitle

\section{Introduction}\label{s:intro}

Let $G$ be a finite group and let $i(G)$ be a group invariant. For example, $i(G)$ could be the set of composition factors of $G$, or the exponent of $G$, or the set of prime divisors of $|G|$, and so on. Given such a property, one can ask whether or not it can be detected from subgroups of $G$ generated by very few elements. For instance, a well-known theorem of Thompson \cite{Thomp} states that $G$ is solvable if and only if every $2$-generated subgroup is solvable.

Let $d \in \mathbb{N}$ be minimal such that $G$ contains a $d$-generated subgroup $H$ with $i(H)=i(G)$. This integer is studied by Lucchini, Morigi and Shumyatsky in \cite{Bound}, where several interesting results are established. For example, they prove that $d=2$ if $i(G) = \pi(G)$ is the set of prime divisors of $|G|$ (cf. Problem 17.125 in \cite{kou}), and $d=3$ if $i(G) = \Gamma(G)$ is the \emph{prime graph} of $G$, which is a graph with vertex set $\pi(G)$, and two vertices $p$ and $q$ are adjacent if and only if $G$ contains an element of order $pq$. They also show that  $d \le 4$ if $i(G) = {\rm exp}(G)$ is the exponent of $G$, and the better bound $d \le 3$ has recently been established by Detomi and Lucchini \cite[Theorem 1.6]{Prob} (determining whether or not $d=2$ in this situation is an open problem). 

The main aim of this paper is to extend the study of boundedly generated subgroups initiated in \cite{Bound} from finite groups to \emph{profinite} groups, with some suitable (and necessary) modifications. 

Recall that a profinite group is a topological group which is an inverse limit of finite groups (which are equipped with the discrete topology). Given a closed subgroup $H$ of a profinite group $G$, its index $|G : H|$ is the least common multiple of the indices of the open subgroups of $G$ containing $H$. Hence, the order of a profinite group $G$ is defined to be $|G : 1|$, which is a \emph{supernatural number} (or \emph{Steinitz number}), that is, a formal infinite product $\prod p^{n(p)}$ over all primes $p$, in which each $n(p)$ is a non-negative integer or infinity. In addition, the order of an element $g \in G$, denoted $|g|$, is defined to be the order of the subgroup topologically generated by $g$, that is, $|g| = |\overline{\la g \ra}|$. If $d$ is a positive integer then a closed subgroup $H$ of $G$ is said to be $d$-generated if $H = \ol{\la x_1, \ldots, x_d \ra}$ for some $x_i \in H$, which is equivalent to the condition that $HN/N = \la x_1 N,\ldots, x_d N \ra$ for every open normal subgroup $N$ of $G$.

Given these definitions, we can consider the exponent of a profinite group $G$ and the set of prime divisors of $|G|$, denoted by $\exp(G)$ and $\pi(G)$, respectively. We can also define the prime graph $\Gamma(G)$. 

Our first result is a natural extension of \cite[Theorem C]{Bound}.

\begin{theoremA}
\emph{Let $G$ be a profinite group. Then there exists a $3$-generated (closed) subgroup $H$ of $G$ such that $\G(H)=\G(G)$.}
\end{theoremA}
	
This result  is best possible. Indeed, there exists a finite group $G$ such that $\G(H) \neq \G(G)$ for every $2$-generated subgroup $H$ of $G$ (see \cite[p.$\,883$]{Bound}). A key tool in the proof of Theorem A is an extension of \cite[Theorem C]{Bound} for finite groups (see Proposition \ref{p.g.gen}): if $1=M_n\leq \cdots \leq M_0=M$ is a normal series of a finite group $M$, then there is a $3$-generated subgroup $K$ of $M$ such that $\Gamma(KM_i/M_i) = \Gamma(M/M_i)$ for all $i$.

In order to prove Proposition \ref{p.g.gen}, we will show that if $M$ has no proper subgroup $K$ with the desired property, then $M$ is $3$-generated (so the conclusion holds with $K=M$). In particular, we are naturally led to consider the minimum number of generators of a finite group. A fundamental role in this investigation is played by the so-called \textit{crown-based powers} of a monolithic group (a finite group is said to be \emph{monolithic} if it has only one minimal normal subgroup). The notion of a \textit{crown} was introduced by  Gasch\"utz \cite{prefratt} in the context of solvable groups, in his construction of \textit{prefrattini subgroups}. More recently, this notion has been generalized to all finite groups (see \cite{Forster}, for example). In \cite{crown}, Detomi and Lucchini introduce crown-based powers as an extension of crowns, where they are used to study the probabilistic zeta functions of finite groups.
For each positive integer $k$, the crown-based power $L_k$ of a monolithic group $L$ is defined as a subgroup of $L^k$ (the $k$-fold direct product of $L$) whose socle is a crown (see Definition \ref{d:cbp}). In \cite{crown}, the authors also give conditions on the minimal generation of homomorphic images of a finite group that imply it is isomorphic to a crown-based power of some monolithic group (see Theorem \ref{G&L}). This result, together with the generating properties of crown-based powers recorded in Section \ref{s:prel}, will be used in the proof of Proposition \ref{p.g.gen}, which in turn plays a key role in the proof of Theorem A. Properties of crown-based powers will be discussed in more detail in Section \ref{s:prel}.

The next result provides a profinite analogue of a recent theorem of Detomi and Lucchini \cite[Theorem 1.5]{Prob} on the prime divisors of indices of subgroups of finite groups. It implies that 
if $C$ is a closed subgroup of a profinite group $G$ then it is possible to get information on the set of prime divisors of $|G:C|$ from the primes that divide $|H : C\cap H|$, where $H$ is an appropriate subgroup of $G$ with very few generators.

\begin{theoremB}
\emph{Let $G$ be a profinite group and $X,C$ be two (closed) subgroups of $G$ such that $X\leq C$. Then there exist $a,b\in G$ such that 
\[\pi(|G : C|) \subseteq \pi(|\langle a,b,X \rangle : C \cap \langle a,b,X \rangle|),\]
where $\pi(n)$ denotes the set of prime divisors of the integer $n$.}
\end{theoremB}
	
Again, this result is best possible; in general, the conclusion does not hold if we only take a single element $a \in G$ (for example, let $G=S_3$ and $X=C=1$). A key result in the proof of Theorem B is Proposition \ref{pdfinitoGEN}, which states that if $1=M_n\leq \cdots \leq M_0=M$ is a normal series of a finite group $M$, and $X\leq C$ are two subgroups of $M$, then there exist $a,b\in M$ such that 
\[\pi(|M : CM_i|) \subseteq \pi(|\langle a,b,X \rangle M_i : CM_i \cap \langle a,b,X \rangle M_i|)\]
for all $i$. In order to prove the existence of such elements, we will bound the minimal number of generators of $M$ with respect to $X$, i.e. the minimum number of elements, which together with $X$, are needed to generate $M$. Crown-based powers will also play an important role in this analysis (see Lemma \ref{maind(L)Y} and Proposition \ref{G&L&X}). As a corollary to Theorem B, we also extend a theorem of Camina, Shumyatsky and Sica \cite[Theorem 1.2]{Cam_sica} from finite groups to profinite groups (see Theorem \ref{t:css}).

Motivation for our final theorem stems from the following general problem: 

\begin{problem}\label{p:1}
Is there a constant $d$ such that every profinite group $G$ contains a $d$-generated (closed) subgroup $H$ such that $\exp(H)=\exp(G)$?
\end{problem}

If $G$ is finite then this problem has a positive solution, with $d \le 3$ (see \cite[Theorem 1.7]{Prob}), and $d=2$ if $G$ is solvable (see \cite[Theorem E]{Bound}). The general problem for profinite groups is more difficult, and it is related to the 
following well-known open problem proposed by Hewitt and Ross in \cite{H&R} (also see \cite[Section 4.8.5]{Ribes}):

\begin{problem}\label{p:2}
Does every torsion profinite group have finite exponent?
\end{problem} 

In \cite{Zel}, Zel'manov proves that every finitely generated torsion pro-$p$ group is finite, but the general problem is still open. In view of Zel'manov's theorem, it is not difficult to show that a positive solution to Problem \ref{p:1} will yield a positive solution to Problem \ref{p:2} (see Section \ref{s:c} for the details). Here, we use Zel'manov's theorem to establish a best possible result for finitely generated prosupersolvable groups.

\begin{theoremC}
\emph{Let $G$ be a finitely generated prosupersolvable group. Then there exists a $2$-generated (closed) subgroup $H$ of $G$ such that $\exp(H)=\exp(G)$.}
\end{theoremC}
	
The main step in the proof is to construct a chain of closed subsets of $G\times G$ such that, by the compactness of $G$, their intersection is non-empty and contains the pair of generators that we are looking for. This chain will be constructed inductively, using a new result on $2$-generated prosolvable groups (see Theorem \ref{2-gen_pro}).

\vspace{2mm}

\noindent \textbf{Notation.} Our notation is fairly standard. Let $G$ be a group and let $n$ be a positive integer. We write $G^n$ for the direct product $G \times \cdots \times G$ ($n$ factors) and $\pi(n)$ for the set of prime divisors of $n$. If $x,y \in G$ and $N$ is a normal subgroup of $G$ then we write $x \equiv y \mbox{ mod }N$ if $Nx = Ny$. If $G$ is finitely generated then $d(G)$ denotes the minimal size of a generating set for $G$, and if $G$ is finite we write ${\rm Soc}(G)$ for the socle of $G$, which is the subgroup generated by the minimal normal subgroups of $G$. Finally, if $G$ is a profinite group and $H$ is a closed subgroup then we define the prime graph $\Gamma(G)$, the index $|G:H|$ and the exponent ${\rm exp}(G)$ as above. Also, if $N$ is an open normal subgroup of $G$ then we denote this by writing $N\unlhd_o G$.

\section{Preliminaries}\label{s:prel}

Here we record some preliminary results that will be useful in the proof of the main theorems. 
Our main references for profinite groups are \cite{Ribes} and \cite{Wilson}. If we refer to a subgroup of a profinite group, then it is assumed to be closed, unless stated otherwise. 

As noted in the Introduction, \emph{crown-based powers} play an important role in the proofs of Theorems A, B and C. In order to give the definition, recall that a \emph{monolithic} group is a finite group with a unique minimal normal subgroup (its socle).

\begin{defn}\label{d:cbp}	
Let $L$ be a monolithic group with socle $A$. If $A$ is abelian, assume also that $A$ has a complement in $L$.
For each positive integer $k$, the \emph{crown-based power} of $L$ of size $k$ is the subgroup $L_k$ of $L^k$ defined by
\begin{displaymath}\label{crown_eq}
L_k = \{(l_1,\ldots,l_k)\in L^k \mid l_1\equiv \cdots \equiv l_k \textrm{ mod } A\}.
\end{displaymath}
\end{defn}
		
Clearly, $\Soc(L_k)$ is a direct product of $k$ minimal normal subgroups (each isomorphic to $A$), and $L_k/\Soc(L_k) \cong L/ \Soc(L)$. In addition, if $T$ is a complement for $A$ in $L$, then the diagonal subgroup
\[\diag(T) = \{ (t,\ldots, t) \mid t\in T\} \leq L_k\]
is a complement for $A^k$ in $L_k$, i.e., $L_k=A^k \diag(T)$. Moreover, it is easy to see that the quotient group of $L_{k+1}$ over any minimal normal subgroup is isomorphic to $L_k$. 

We will need the following result on the normal structure of crown-based powers (see \cite[Lemma 2.5]{Bound}).
 
\begin{prop}\label{socN}
Let $L$ be a monolithic group and let $L_k$ be the crown-based power of $L$ of size $k$. If $N$ is a normal subgroup of $L_k$, then either $\Soc(L_k) \leq N$ or $N \leq \Soc(L_k)$.
\end{prop}

Let $G$ be a finitely generated group and let $d(G)$ be the minimal size of a generating set for $G$. Since the quotient group of $L_{k+1}$ over any minimal normal subgroup is isomorphic to $L_k$, the sequence $d(L_1),d(L_2), \ldots$ is increasing. Moreover, the main theorem of \cite{gener} implies that $d(L_{k+1}) \leq d(L_k) + 1$.  The next result determines $d(L_k)$ precisely in the cases we will be interested in. (In part (ii), $H^1(G,V)$ denotes the first cohomology group of $G$ on a $G$-module $V$, and $\End_{G}(V)$ is the group of endomorphisms of $V$ as a $G$-module.)

\begin{prop}\label{Maind(L_k)}
Let $L$ be a non-cyclic monolithic group with socle $A$, and let $L_k$ be the crown-based power of $L$ of size $k$.
\begin{itemize}
\item[{\rm (i)}] $d(L_1) = d(L) = \max\{2, d(L/A)\}$;
\item[{\rm (ii)}] Assume that $A$ is abelian and complemented in $L$. Write $q=|\End_{L/A}(A)|$, $q^r = |A|$, $q^s=|H^1(L/A,A)|$ and set $\theta=0$ or $1$ according to whether or not $A$ is a trivial $L/A$-module. Then
\[d(L_k) = \max\{d(L/A),\theta +\lceil(k+s)/r \rceil\},\]
where $\lceil x \rceil$ denotes the smallest integer greater than or equal to $x$.
\end{itemize}
\end{prop}

\begin{proof}
Part (i) is a special case of \cite[Theorem 1.1]{men}, and part (ii) is a consequence of \cite[Proposition 6]{Morini}.
\end{proof}

For the purposes of this paper, the main connection between arbitrary finite groups and crown-based powers is provided by the following theorem (see \cite[Theorem 1.4]{L_k}).

\begin{thm}\label{G&L}
Let $m$ be a positive integer and let $G$ be a finite group such that $d(G/N) \leq m$ for every non-trivial normal subgroup $N$, but $d(G) >m$. Then there exists a monolithic group 
$L$ such that $G \cong L_k$ for some $k$, and $\Soc(L)$ is either non-abelian or complemented.
\end{thm}

We close this section by giving a brief overview of our general strategy, which will explain the relevance of Theorem \ref{G&L} (strictly speaking, we use a variant of this approach in the proof of Theorem C). Let $G$ be a profinite group and let $i(G)$ be a group invariant. We begin with an initial reduction to the case where $G$ is countably based, so we can consider a chain 
\[G = M_0 \ge M_1 \ge \cdots\]
of open normal subgroups of $G$, which form a base of open neighbourhoods of $1$. 

Next, by choosing $d \in \mathbb{N}$ appropriately (i.e. $d=3$ in the proof of Theorem A, and $d=2$ in Theorem B), we define a collection of closed subsets of $G^d$,
\[\Omega _j = \{(x_1, \ldots, x_d) \in G^d \mid i(\la x_1M_j, \ldots, x_dM_j\ra) = i(G/M_j)\}\]
for each $j=0,1,2,\ldots$, and we note that $\Omega _j$ is non-empty by known results for finite groups. For example, we can appeal to \cite[Theorem C]{Bound} if $i(G) = \Gamma(G)$ and $d=3$. The key step is to establish the existence of an element
\[(x_1, \ldots, x_d) \in \bigcap_{j \ge 0}\Omega _j.\]
Indeed, given such an element we can show that $i(C) = i(G)$, where $C = \overline{\la x_1, \ldots, x_d\ra}$, and this is how the main theorems are proved.

In order to prove that this intersection is non-empty, it suffices to show (by the compactness of $G$) that every finite subcollection of the $\Omega _i$ has non-empty intersection. In turn, this problem can be stated in terms of finite groups as follows: given a normal series 
\[1 = T_n \le T_{n-1} \le \cdots \le T_1 \le T_0 = T\]
of a finite group $T$, it suffices to show that there exists a $d$-generated subgroup $S$ of $T$ such that  $i(ST_j/T_j) = i(T/T_j)$ for all $j$. To establish the existence of $S$, we assume that there is no \emph{proper} subgroup $S$ of $T$ with the desired property, and our aim is to show that $d(T) \le d$ (so we can take $S=T$). 

Given such a group $T$, there exists a normal subgroup $N$ of $T$ such that $d(T/N) = d(T)$ and every proper quotient of $T/N$ can be generated by $d(T)-1$ elements. At this point, we can apply Theorem \ref{G&L}, which implies that $T/N \cong L_k$ for some monolithic group $L$ and an integer $k \ge 1$. Our aim now is to show that $k \le t$ for some $t$, which depends on the invariant we are considering (for example, if $i(G) = \Gamma(G)$ then the proof of Proposition \ref{p.g.gen} reveals that $k \le 3$). We can then appeal to Proposition \ref{Maind(L_k)} to compute $d(L_t)$, which gives an upper bound for $d(T)$. In this way, we can show that $d(T) \le d$, as required.

\section{Proof of Theorem A}\label{s:A}

Let $G$ be a profinite group and let $\pi(G)$ be the set of primes dividing $|G|$, the order of $G$. Recall that $\Gamma(G)$ denotes the \emph{prime graph} of $G$. This is a graph with vertex set $\pi(G)$, and two vertices $p$ and $q$ are adjacent if and only if $G$ has an element of order $pq$.

By a theorem of Lucchini, Morigi  and Shumyatsky \cite[Theorem C]{Bound}, if $G$ is finite then there exists a $3$-generated subgroup $H$ of $G$ such that $\G(H)=\G(G)$. Moreover, this is best possible. Indeed, there is a $3$-generated finite group $G$ such that $\Gamma(H) \neq \Gamma(G)$ for every $2$-generated subgroup $H$ of $G$ (see \cite[p.883]{Bound}).

In order to extend this result to profinite groups, we will establish the following generalization of \cite[Theorem C]{Bound}, which we can readily apply in the profinite case.

\begin{prop}\label{p.g.gen}
Let $G$ be a finite group and let $1=M_n \le \cdots \le M_1 \le M_0=G$ be a normal series for $G$. Then there exists a $3$-generated subgroup $H$ of $G$ such that $\Gamma(HM_i/M_i)=\Gamma(G/M_i)$ for all $i$.
\end{prop}
	
To prove this result, we need the following two lemmas. In the statement of Lemma \ref{L_k3-gen}, recall that a finite group $L$ is \emph{almost simple} if $S \leq L \leq {\rm Aut}(S)$ for some finite non-abelian simple group $S$. Note that every almost simple group is monolithic.

\begin{lem}\label{quaternion}
Let $p$ and $q$ be primes, and let $P$ be a $p$-group acting on a $q$-group $Q$ such that  $C_Q(a)=1$ for all $1 \neq a \in P$. Then either $P$ is cyclic, or $p=2$ and $P$ is generalized quaternion.
\end{lem}

\begin{proof}	
This is \cite[Proposition 10.3.1(iv)]{Gores}.
\end{proof}

\begin{lem}\label{L_k3-gen}
Let $L$ be an almost simple group. If $k \leq 2$ then $d(L_k) \le 3$.
\end{lem}

\begin{proof}
See Lemma 3.3 in \cite{Bound}.
\end{proof}

In what follows, given two graphs $\G_1$ and $\G_2$, we say that $\G_1 \subseteq \G_2$ if and only if they have the same set of vertices, and if $x,y$ are connected by an edge in $\G_1$, then they are also adjacent in $\G_2$.

\begin{proof}[Proof of Proposition \ref{p.g.gen}]
%
First observe that it suffices to prove the following claim:

\vs

\noindent \textbf{Claim 1.} \emph{If $G$ has no proper subgroup $H$ such that $\G(HM_i/M_i)=\G(G/M_i)$ for all $i$, then $d(G) \le 3$.}

\vs

\noindent Clearly, in this situation, the conclusion to Proposition \ref{p.g.gen} holds with $H=G$. To see that Claim 1 is sufficient, suppose that $G$ has a proper subgroup $H$ such that $\G(HM_i/M_i)=\G(G/M_i)$ for all $i$. If $d(H) \le 3$ then we are done, so assume $d(H)>3$. Set $T_i = H \cap M_i$ for each $i$, so 
$1 = T_n \le \cdots \le T_1 \le T_0=H$ is a normal series for $H$. Then $H$ must have a proper subgroup $K$ such that $\Gamma(KT_i/T_i) = \Gamma(H/T_i)$ for all $i$, otherwise Claim 1 implies that $d(H) \le 3$, which is not the case. It is easy to see that $\Gamma(KM_i/M_i) =\Gamma(G/M_i)$ for all $i$, so we are done if $d(K) \le 3$. If $d(K)>3$, then we can repeat the process, with $K$ in place of $H$. In this way, since $G$ is finite, we will eventually find a $3$-generated subgroup of $G$ with the desired property, so Claim 1 is indeed sufficient to prove the proposition.

\begin{proof}[Proof of Claim 1]

Clearly, we may assume that $G$ is non-cyclic. Let $N$ be a normal subgroup of $G$ such that $d(G/N)=d=d(G)$, but every proper quotient of $G/N$ can be generated with $d-1$ elements. Then Theorem \ref{G&L} implies that 
$G/N\cong L_k$ for a suitable $k$, where $L$ is a monolithic group and $\Soc(L)=A$ is non-abelian or complemented. 

Consider the series
\[N / N\leq M_{n-1}N/N\leq\cdots\leq M_iN/N\leq\cdots\leq G/N\cong L_k.\]
We can refine this series to obtain a chief series of $G/N$, and it suffices to show that the desired conclusion holds for this series. 

By Proposition \ref{socN}, there exist integers $h,j$ such that $0\leq h < j \leq n$ and
\[M_hN/N\cong\textrm{Soc}(G/N)=\textrm{Soc}(L_k), \;\; M_hN/M_jN\cong \textrm{Soc}(L),\;\; G/M_jN\cong L.\] 
So for each $i$, either $G/M_iN \cong L_{t_i}$ for some $1\leq t_i\leq k$, or $G/M_iN$ is isomorphic to a quotient of $L$. More precisely: 
\begin{itemize}
\item If $i<j$ then $G/M_iN\cong (G/M_jN)/(M_iN/M_jN)$ is isomorphic to a quotient of $L$.  
\item if $i\geq j$ then
\[M_iN / N \leq \Soc( G / N) \cong \Soc(L_k) = A^k\]
and thus $M_iN / N \cong A^{u_i}$ with $u_i \leq k$. This implies that 
\[G / M_iN \cong L_k / A^{u_i} \cong L_{t_i}\]
with $t_i + u_i = k$. 
\end{itemize}

Fix an isomorphism $\varphi: L_k\rightarrow G/N$ and write $\Soc(L_k) = A_1 \times \cdots \times A_k$, a direct product of minimal normal subgroups (each isomorphic to $A$). We can choose the $A_r$ so that 
\[M_iN/N = \varphi \Bigg( \prod_{r>t_i}A_r \Bigg)\]
if $i\geq j$. 
For all $i\leq k$, define
\[T_i=\{(l_1,\ldots,l_k)\in L_k \mid  l_r=l_i \mbox{ for all } r>i \}\leq L_k\]
Note that $T_i\cong L_i$. We use the $T_i$ to define a subgroup $X \le G$ in the following way:
\begin{itemize}
\item If $k=1$ or if $A$ is not simple, let $X/N=\varphi(T_1)$;
\item If $k=2$ or $A$ is simple and $|A|>2$, let $X/N=\varphi(T_2)$;
\item Otherwise $|A|=2$, $k\geq 3$ and we let $X/N=\varphi(T_3)$.
\end{itemize}
Note that $X/N\cong L_u$ for some $u\leq k$. Our aim is to prove the following claim:

\vspace{2mm}

\noindent \textbf{Claim 2.} $X=G$. 
 
 
\begin{proof}[Proof of Claim 2]
If $k=1$, or if $k=2$ and $A$ is simple, then the claim is obviously true. Therefore, we may assume that the following hypothesis holds:
\begin{equation}\label{hyp}
\mbox{$k\geq 2$, and $A$ not simple if $k=2$.}
\end{equation}
We will prove that $\Gamma(XM_i/M_i)=\Gamma(G/M_i)$ for every $i$. Since $N\leq X$, it suffices to prove that  $\Gamma(XM_i/M_iN)=\Gamma(G/M_iN)$ then, using the assumption that $G$ has no proper subgroups satisfying this equality, we will deduce that $X=G$. We distinguish two cases, according to the value of $i$.

\vspace{2mm}
		
\noindent \textit{Case 1.} $i<j$ 

\vspace{2mm}

Here $G/M_iN$ is isomorphic to a quotient of $L$, and $\varphi(\Soc(L_k)) \leq M_iN/N$. Moreover, since $T_i\Soc(L_k) = L_k$ by definition of $T_i$, we get
\[G/N  = \varphi (L_k) = \varphi (T_i \Soc(L_k) ) = \varphi(T_i)\varphi(\Soc(L_k)) \leq XM_iN / N =XM_i / N\]
and thus $G=XM_i$, so $\G(G/M_i) = \G(XM_i / M_i )$ as required.

\vspace{2mm}

\noindent \textit{Case 2.}  $i\geq j$ 

\vspace{2mm}

Here $G/M_iN\cong L_{t_i}$ for some $1\leq t_i \leq k$. Set  
\[\ol{X}=XM_i/M_i, \;\; \ol{G}=G/M_i, \;\; \ol{N}=M_iN/M_i\]
and note that
\[\ol{G}/\ol{N} \cong (G/M_i) / (M_i N/M) \cong G/M_iN\cong L_{t_i}.\]
Our aim is to show that $\G(\ol{X})=\G(\ol{G})$, or equivalently $\G(\ol{X} / \ol{N} )=\G(\ol{G} / \ol{N} )$.
		
Since $X/N\cong L_u$ we have 
\[|G/N| = |L_k| = |L/A| |A|^k,\;\; |X/N| = |L_u| = |L/A| |A|^u\]
and thus $\pi(X) = \pi(G)$. In particular, $\pi(\ol{X})=\pi(\ol{G})$, so the graphs $\G(\ol{X})$ and $\G(\ol{G})$ have the same set of vertices.  It remains to prove that they also have the same edges.
		
If $u \geq t_i$, then 
\[\varphi (A_{u+1} \times \cdots \times A_{k}) \leq \varphi (A_{{t_i}+1} \times \cdots \times A_{k}) = M_iN / N.\]
Moreover, $T_u(A_{u+1} \times \cdots \times A_{k}) = L_k$ so 
\[G/N = \varphi(T_u(A_{u+1} \times \cdots \times A_{k}))\leq XM_iN/N= XM_i/N\]
and thus $G=XM_i$ and $\G(G/M_i) = \G(XM_i / M_i )$.
		
Now assume $u<t_i$. Set $V:= A_{{t_i}+1} \times \cdots \times A_k$ and note that $\varphi(V) =M_iN/N$ since $i\geq j$. Also note that $L_k/V\cong L_{t_i}$. Now $\varphi :L_k \rightarrow G/N$ induces an isomorphism 
\[\psi: L_k/V \rightarrow G/M_iN.\] 
Note that $T_uV / V \cong T_u$ since $T_u \cap V =1$ and 
\[X /N = \varphi(T_u) \cong \psi(T_uV/V) =XM_i/M_iN.\]
Clearly, $\G(\ol{X}) \subseteq \G(\ol{G})$. For the reverse inclusion,  suppose that $p,q\in \G(\ol{G})$ are connected, so there exists $\ol{g}\in\ol{G}$ such that $|\ol{g}|=pq$. If $\ol{g}\in\ol{ N}\leq\ol{ X}$, then $p$ and $q$ are connected in $\G(\ol{ X})$, and we are done. Therefore, we may assume that $\ol{ g}\in\ol{ G}\setminus \ol{ N}$, so $|\ol{ g}\ol{ N}| \in \{pq,p,q\}$. 

Suppose first that $\ol{ g}\ol{ N}$ has order $pq$, in which case $pq$ divides $|L|$. If $l \in L$ has order $pq$, then $\psi((l,\ldots,l)V)\in XM_i/M_iN$ has order $pq$, hence there is an edge in $\G(\ol{X} / \ol{N})=\G(XM_i/M_iN)$ connecting $p$ and $q$. Now assume $L$ has no elements of order $pq$, so $k>1$.  Consider the preimage of $\ol{ g}\ol{ N}$ in $G/N$. By  applying $\varphi^{-1}$, we obtain $(l_1,\ldots,l_k)\in L_k$ which has order divisible by $pq$, so there exist $l_r, l_s$ such that $p$ divides $|l_r|$ and $q$ divides $|l_s|$. As $|l_rA|=|l_sA|$, it follows that $l_r,l_s\in A$, so the order of $A$ is divisible by two different primes. But  since $A$ has no element of order $pq$, it follows that $A$ is simple and thus $X/N=\varphi(T_2)$ (by definition of $X$). Then the order of $\psi((l_r,l_s,\ldots,l_s)V)\in XM_i / M_iN$ is  divisible by $pq$, and thus $p$ and $q$ are connected in $\G(\ol{X} / \ol{N} )$.
		
Now suppose that $\ol{ g}\ol{ N}$ has order $p$ or $q$. Without loss of generality, we may assume that the order is $p$, so $q$ divides $|\ol{N}|$. We consider two cases.

Suppose that $p\nmid |A|$. By the definition of $X$, every Sylow $p$-subgroup of $X/N$ is a Sylow $p$-subgroup of $G/N$, hence the same holds for $\ol{X}/\ol{N}$. As $\ol{ g}\ol{ N}$ has order $p$, there exists $\ol{h}\in\ol{G}$ such that $(\ol{ g}\ol{ N})^{\ol{h}}\in\ol{X}/\ol{N}$, so $\ol{g}^{\ol{h}}\in \ol{X}$ is an element of order $pq$, and thus $p$ and $q$ are connected in $\Gamma(\ol{X})$. 

Now assume that $|A|$ is divisible by $p$. Let $\ol{Q}\in {\rm Syl}_p(\ol{N})$ and let $\ol{T}=N_{\ol{X}}(\ol{Q})$, hence $\ol{X}=\ol{T}\,\ol{N}$ by the Frattini argument. Let $\ol{P}\in {\rm Syl}_p(\ol{T})$, which acts on $\ol{Q}$ by conjugation. Since \eqref{hyp} holds, a Sylow $p$-subgroup of $\ol{X}/\ol{N}$ is non cyclic, hence $\ol{P}$ is non-cyclic. Moreover, if $p=2$, then $\textrm{Soc}(\ol{X}/\ol{N})$ contains a subgroup isomorphic to an elementary abelian $2$-group $K$ of rank 3, so $\ol{P}$ has a section isomorphic to $K$ and thus $\ol{P}$ is not generalized quaternion. By Lemma \ref{quaternion}, there is a non-trivial element in $\ol{Q}$ which is centralized by some non-trivial element from $\ol{P}$. Hence, there is an element of order $pq$ in $\ol{X}$.
		
We have now shown that $\G(XM_i/M_i)=\G(G/M_i)$ for any $i=1,\ldots,t$. Therefore, by the assumption on $G$, we can conclude that $X=G$ and the proof of Claim 2 is complete.
\end{proof}

			
Finally, we will bound $d(G/N)=d(G)=d$ using the fact that $G/N = X/N \cong L_k$. There are $4$ cases to consider:
\begin{itemize}
\item [1.] If $k=1$ or $A$ is not simple, then $G/N=\varphi(T_1)\cong L$ so Proposition \ref{Maind(L_k)}(i) yields
\[d(G/N)=d(L) = \max\{2,d(L/A)\}\leq\max\{2,d-1\}\]
and thus $d=2$.
\item [2.] If $A$ is abelian and $|A|>2$, then $k\leq2$. Let  $q,r,s,\theta$ be as in Proposition \ref{Maind(L_k)}(ii). Since $d(L/A)\leq d-1$, it follows that
\[d=d(G/N)=\theta+\lceil(k+s)/r\rceil\leq1+\lceil(2+s)/r\rceil\]
so we have $d\leq3$ since $s<r$ (this follows from an important result of Aschbacher and Guralnick, see \cite[Theorem A]{Coho}).
\item [3.] If $|A|=2$, then $k\leq3$ and we deduce that $d\leq3$ as in the previous case (note that $\theta=0$).
\item [4.] In the remaining cases, $A$ is a non-abelian simple group and $k=2$. By Lemma \ref{L_k3-gen} and the definition of $X$, $G/N\cong L_2$ is $3$-generated, so $d\leq3$.
\end{itemize} 

This completes the proof of Claim 1.
\end{proof}

In view of our earlier comments (see the paragraph following Claim 1), this completes the proof of Proposition \ref{p.g.gen}.
\end{proof}

The next result provides a useful characterization of the prime graph of a profinite group. Note that, given two (or more) graphs $\G_1$ and $\G_2$, we define $\G = \G_1 \cup \G_2$ to be  the  graph whose set of vertices is the union of the vertices of $\G_1$ and $\G_2$, and two vertices $x$ and $y$ are adjacent in $\G$ if and only if they are adjacent in $\G_1$ or $\G_2$. 

\begin{prop}\label{caracp.g.}
Let $G$ be a profinite group. Then there exists a countable set $\mathcal{N}$ of open normal subgroups of $G$ such that 
\[\G(G)=\bigcup_{N\in\mathcal{N}}\G(G/N).\]
\end{prop}

\begin{proof}
If $p\in\G(G)$ then $p$ divides $|G|$, which we recall is defined by  
\[|G|= \textrm{lcm}\{|G:N| \mid N\unlhd_o G \}.\]
Therefore, choose an open normal subgroup $N_p$ of $G$ such that $p$ divides $|G:N_p|$, so $p\in\G(G/N_p)$. Suppose $p,q\in\G(G)$ are connected, so there is an element $g\in G$ such that $pq$ divides $|g| = |\overline{\la g\ra}|$. Hence, choose  $N_{pq}\unlhd_o G$ such that $pq$ divides $|\la g\ra N_{pq}/N_{pq}|$, and thus $p$ and $q$ are connected in $\G(G/N_{pq})$.  Let $\mathcal{N}$ be the set of all such chosen open normal subgroups of $G$; since we have a countable number of primes dividing the order of $G$, the set $\mathcal{N}$ is countable.

By this construction, $\G(G) \subseteq \bigcup_{N\in\mathcal{N}}\G(G/N)$. The reverse inclusion also holds since $\G(G / N) \subseteq \G(G)$ for every open normal subgroup $N$ of $G$.
\end{proof}
		
We are now in a position to prove Theorem A.


\begin{proof}[Proof of Theorem A]
Let $\mathcal{N}$ be the set of open normal subgroups of $G$ given by Proposition \ref{caracp.g.} and set $M=\bigcap_{N\in\mathcal{N}}N$, so $G/M$ is a profinite group. The family $\{N/M\}_{N\in\mathcal{N}}$ is a countable basis of open subgroups of $G/M$. By taking appropriate intersections between members of this family, we can choose a basis of open subgroups of $G/M$ which are totally ordered with respect to inclusion. For this reason, we may assume that $G$ is countably based, which means that $G$ has a chain 
\begin{equation}\label{e:chain}
G=M_0 \ge  M_1\ge M_2 \ge \cdots 
\end{equation}
of open normal subgroups comprising a base of open neighbourhoods of $1$ (see \cite[Proposition 4.1.3]{Wilson}). For every open normal subgroup $M_i$ in the chain, let us define
\[\Omega_i=\{(x_1,x_2,x_3)\in G^3 \mid \G(\langle x_1,x_2,x_3 \rangle M_i/M_i)=\G(G/M_i)\}.\]
		
Note that if $(x_1,x_2,x_3)\in\Omega_i$ then $x_1M_i\times x_2M_i\times x_3M_i\subseteq\Omega_i$. In fact, $\Omega_i$ is the union of finitely many subsets of this form. Since each $M_i$ is closed in $G$, it follows that $x_1M_i\times x_2M_i\times x_3M_i$ is closed in $G\times G\times G$, so $\Omega_i$ is also closed.

Choose a finite subchain $M_{i_1} \leq \cdots \leq M_{i_r}$ of \eqref{e:chain} and consider the series 
\[M_{i_1}/M_{i_1} \leq M_{i_2}/M_{i_1}\leq\cdots\leq M_{i_r}/M_{i_1}\leq G/M_{i_1},\]
which is a normal series for $G/M_{i_1}$. Since $G/M_{i_1}$ is finite, Proposition \ref{p.g.gen} implies that there exists $(y_1,y_2,y_3) \in G^3$ such that $\G(\langle y_1,y_2,y_3 \rangle M_{i_1}/M_{i_1}) = \G(G/M_{i_1})$ and for any $M_{i_j}/M_{i_1}\unlhd G/M_{i_1}$ we have 
\[\G(\langle y_1,y_2,y_3 \rangle M_{i_j}/M_{i_j})=\G(G/M_{i_j})\]
with $j=2,\ldots,r$. This means that $(y_1,y_2,y_3)\in\bigcap_{j=1}^{r}\Omega_{i_j}$, so every finite subcollection of the $\Omega_i$ has non-empty intersection. Since $G$ is compact, the whole family has non-empty intersection, which implies that there exists $(x_1,x_2,x_3)\in G^3$ such that 
\[\G(\langle x_1,x_2,x_3 \rangle M_i/M_i) = \G(G/M_i)\] 
for all $M_i$ in \eqref{e:chain}.

Let $C = \ol{\langle x_1,x_2,x_3 \rangle}$. We want to show that $\G(C)=\G(G)$. Clearly, $\G(C)\subseteq \G(G)$. If $p\in\G(G)$ then there exists $N_p\in\mathcal{N}$ such that $p$ divides $|G / N_p|$, so there exists a subgroup $M$ in the chain \eqref{e:chain} such that $M\leq N_p$. Then 
\[ p\in\G(G/N_p)\subseteq\G(G/M)  =\G(\langle x_1,x_2,x_3 \rangle M/M) 
				 = \G(CM/M)=\G(C/C\cap M)\subseteq\G(C).\]
		
We can use the same argument to prove that $\G(G)$ and $\G(C)$ have the same edges. Indeed, suppose $p,q\in \G(G)$ are connected. Then there exists $N_{pq}\in \mathcal{N}$ such that $p$ and $q$ are connected in $\G( G / N_{pq})$, and there exists a subgroup $M'$ in \eqref{e:chain} such that $M'\leq N_{pq}$. As before, we get
\[\G(G/N_{pq})\subseteq\G(G/M')  =\G(\langle x_1,x_2,x_3 \rangle M'/M')
		 = \G(CM'/M')=\G(C/C\cap M')\subseteq\G(C).\]
Therefore $p$ and $q$ are connected in $\G(C)$, and we conclude that $\G(C)=\G(G)$.
\end{proof}

\section{Proof of Theorem B}\label{s:B}

In this section we turn our attention to Theorem B. We start by recording a couple of preliminary results.

Let $X$ be a subset of a finite group $G$, and let $d_X(G)$ be the minimal integer $d$ such that $G=\la X, g_1,\ldots,g_d\ra$ for some $g_1,\ldots,g_d\in G$. The following lemma (see \cite[Lemma 5]{Prdiv}), bounds $d_Y(L)$ for a monolithic group $L$ and $Y\subseteq L$.
 
\begin{lem}\label{maind(L)Y}
Let $L$ be a monolithic group with socle $A$, and let $Y \subseteq L$. Then 
\[d_Y(L) \leq \max\{2, d_{AY}(L)\}.\]
\end{lem}

The next result can be viewed as a generalization of Theorem \ref{G&L}. Here, if $U$ is a subgroup of $L$, then $\diag(U)=\{(u,\ldots,u) \mid u\in U\} \le L_k$. 

\begin{prop}\label{G&L&X}
Let $X$ be a subgroup of a finite group $G$ and let $N$ be a normal subgroup of $G$ such that $N$ is maximal with the property that $\displaystyle{d_{XN}(G) = d_X(G)}$. Then there exists a monolithic group $L$ and an isomorphism $\varphi : G/N \rightarrow L_k$ such that 
$\varphi(X) \leq \diag(L)$. Moreover, if $A$ is abelian and $T$ is a complement of $A$ in $L$, then $\varphi(X)\leq \diag(T)$.
\end{prop}

\begin{proof}
This is a consequence of Proposition 12, Theorem 20 and Corollary 21 of \cite{X-dic}.
\end{proof}

The main motivation for Theorem B is a result of Detomi and Lucchini \cite[Theorem 1.5]{Prob}, which states that if $G$ is a finite group and $X,C$ are two subgroups of $G$ such that $X\leq C$, then it is possible to find two elements $a,b\in G$ such that 
\[\pi(|G:C|) \subseteq \pi(|\la a,b,X \ra : C \cap \la a,b,X \ra|),\] 
where we recall that $\pi(n)$ is the set of primes divisors of the integer $n$. In a similar spirit to the proof of Theorem A, we will prove a generalization of this result (see Proposition \ref{pdfinitoGEN}) which can be applied in the profinite case to establish Theorem B. In order to do this, we require some additional results.

\begin{lem}\label{lemmabasedivisors}
Let $G$ be a finite group and let 
\[1=M_n \le \cdots \le M_i\le \cdots \le M_0=G\] 
be a chief series of $G$. Let $X$ and $C$ be two subgroups of $G$ such that $X\leq C$. For any index $i$, let $\displaystyle{\mathcal{Q}_i \subseteq \pi(|G : CM_i|)}$ and assume that $G$ has no proper subgroup $H$ containing $X$ such that $\mathcal{Q}_i \subseteq \pi(|HM_i : HM_i \cap CM_i|)$. Assume also that there exists an epimorphism $\displaystyle{\varphi : G \rightarrow L_t}$, where $L$ is a monolithic group, $t$ is a positive integer and $\displaystyle{\varphi(X)\leq {\rm diag}(L)}$. Moreover, if the socle $A$ of $L$ is abelian and $T$ is a complement of $A$ in $L$, then assume that $\varphi(X)\leq {\rm diag}(T)$. Then the following hold:
\begin{itemize}
\item[(i)] If $A$ is abelian, then $t=1$;
\item[(ii)] If $A$ is non-abelian, then $t\leq\beta+1$, where $\beta = |\mathcal{Q} \cap \pi(A)|$ and $\mathcal{Q}=\bigcup_{i=1}^{n}\mathcal{Q}_i$.
\end{itemize}
\end{lem}

\begin{proof}
Let $N=\ker(\varphi)$ and set $\overline{G} = G/N\cong L_t$. For the remainder of the proof, we will use the ``bar convention" to indicate the image of a subgroup of $G$ in $\overline{G}$.
		
Define $K \le G$ so that 
\[\overline{K}=
\begin{cases}
\textrm{diag}(L)	& \textrm{if }  A \textrm{ is non-abelian}\\
\textrm{diag}(T)	& \textrm{if }  A \textrm{ is abelian}
\end{cases}\]
In particular, the hypothesis $\varphi(X)\leq \textrm{diag}(L)$ implies that $X\leq K$.
		
The series
\[N / N \leq M_nN/N\leq\cdots\leq M_iN/N\leq\cdots\leq G/N\cong L_t\]
is a chief series for $G/N$. By Proposition \ref{socN}, there exists $0\leq h \leq n$ such that 
\[\overline{M_h}=M_hN/N=\Soc(\overline{G})=\overline{S_1}\times\cdots\times\overline{S_t},\] where each $\overline{S}_i$ is a minimal normal subgroup of $L_t$ isomorphic to $A$.
 		
Recall that if $p$ is a prime, $|G|_p$ denotes the order of a Sylow $p$-subgroup of $G$. For any index $i$, we define
\[\mathcal{P}_i = \{p\in\mathcal{Q}_i \mid p \textrm{ divides } |\overline{G}:\overline{CM_i}| \textrm{ and } |\overline{M_h}\cap \overline{CM_i}|_p<|\overline{M_h}|_p \}	\]
and
\[\mathcal{P}^*_i = \{p\in\mathcal{Q}_i \mid p \textrm{ divides } |\overline{G}:\overline{CM_i}| \textrm{ and } |\overline{M_h}\cap \overline{CM_i}|_p=|\overline{M_h}|_p \}.\]	
Note that $\mathcal{P}_i \subseteq \pi(A)$ for any $i$, hence $\bigcup_{i=1}^n\mathcal{P}_i\subseteq\pi(A)$. If $p\in\mathcal{P}_i$ then there exists $1\leq j_p \leq t$ such that $|\overline{S_{j_p}} \cap \overline{CM_i}|_p < |\overline{S_{j_p}}|_p$.
 		 
Let $\Lambda=\{j_p \mid p \in \bigcup_{i=1}^n \mathcal{P}_i\}$ and  define $\overline{H}=\overline{K}\prod_{l\in \Lambda}\overline{S_l}$. Note that
\[|\Lambda| \leq \Bigg|\bigcup_{i=1}^n \mathcal{P}_i\Bigg| \leq |\mathcal{Q} \cap \pi(A)| = \beta.\]
If $A$ is non-abelian, then either $t=|\Lambda|$ or $\overline{H}\cong L_{|\Lambda|+1}$. If $A$ is abelian, then $A$ is an elementary abelian $p$-group for some prime $p$, so $\pi(A)=\{p\}$ and $|\mathcal{P}_i|\leq1$ for each $i$, hence $|\Lambda|\leq1$ and
\[\overline{H}\cong
\begin{cases}
T & \textrm{if }  |\Lambda|=0, \\
L & \textrm{if }  |\Lambda|=1.
\end{cases}\]
We claim that $\mathcal{P}_i\subseteq \pi(|HM_i : HM_i \cap CM_i|)$ for any $i$.

Assume the claim is false, so there exists $p\in \mathcal{P}_i$ such that $p\notin \pi(|HM_i : HM_i \cap CM_i|)$, and there exists a Sylow $p$-subgroup $U$ of $HM_i$ which is contained in $CM_i$. Then $\overline{U}\cap\overline{S_{j_p}}$ is a Sylow $p$-subgroup of $\overline{S_{j_p}}$, but this contradicts the fact that $|\overline{S_{j_p}} \cap \overline{CM_i}|_p < |\overline{S_{j_p}}|_p$. This justifies the claim. 

Next we claim that $\mathcal{P}^*_i\subseteq \pi(|HM_i : HM_i \cap CM_i|)$ for any $i$. 
		
To see this, let $p\in\mathcal{P}^*_i$, so $p$ divides $|\overline{G} : \overline{CM_i}| = |\overline{G} : \overline{CM_i}\,\overline{M_h}||\overline{CM_i}\,\overline{M_h} : \overline{CM_i}|$, but $p$ does not divide $|\overline{M_h} : \overline{M_h}\cap\overline{CM_i}|=|\overline{CM_i}\,\overline{M_h} : \overline{CM_i}|$.  Hence $p$ must divide $|\overline{G} : \overline{CM_i}\,\overline{M_h}|$, which divides $|G : CM_i\,M_h|$. Since $G=KM_h=HM_h$, it follows that $p$ divides $|HM_h : CM_i\,M_h|$. If $i< h$ then $M_h \subset M_i$, so $p$ divides $|HM_i : CM_i|$ and thus $p$ divides $|HM_i : CM_i \cap HM_i|$. Otherwise, $i\geq h$ and $M_i \subseteq M_h$, hence 
\[|HM_h : CM_iM_h|=|HM_iM_h : CM_iM_h|=|HM_iM_h : CM_iM_h \cap HM_iM_h|\]
and we see that $p$ divides $|HM_i : HM_i \cap CM_i|$. This justifies the claim.
	
Finally, let $p\in\mathcal{Q}_i \setminus \mathcal{P}_i \cup \mathcal{P}^*_i$. Then $p$ divides 
\[|G:CM_i|=|G:CM_iN||CM_iN:CM_i|\]
but $p$ does not divide $|G:CM_iN|$ because $p\notin \mathcal{P}_i \cup \mathcal{P}^*_i$. So $p$ must divide 
\[|CM_iN : CM_i|=|M_iN : M_iN \cap CM_i|.\] 
Let $U$ be a Sylow $p$-subgroup of $HM_i$. Then $U \cap M_i N$ is a Sylow $p$-subgroup of $M_iN$ and it cannot be contained in $CM_i$. This implies that $p\in \pi(|HM_i : HM_i \cap CM_i|)$ for any $i$.
	
We have now shown that $\mathcal{Q}_i \subseteq \pi(|HM_i : HM_i \cap CM_i|)$ for any $i$, so the hypothesis of the lemma implies that $H=G$. 
	
If $A$ is abelian, then either $\overline{H}\cong L$ and thus $\overline{G}=\overline{H}\cong L$ and $t=1$, or $\overline{G}=\overline{H}=\overline{K}\cong T$. However, the latter situation is incompatible with the fact that $L_t$ is an epimorphic image of $G$. If $A$ is non-abelian, then  either $t=|\Lambda|$ or $\overline{G}=\overline{H} \cong L_{|\Lambda|+1}$, so 
$t\leq |\Lambda|+1 \leq \beta+1$. This concludes the proof of the lemma.
\end{proof}


\begin{prop}\label{in mezzo}
Let $G$ be a finite group and let 
\[1=M_n \le \cdots \le M_i\le \cdots \le M_0=G\]
be a chief series of $G$. Let $X,C$ be two subgroups of $G$ such that $X\leq C$.
For any index $i$, let $\mathcal{Q}_i \subseteq \pi(|G : CM_i|)$. Assume that $G$ has no proper subgroup $H$ containing $X$ such that $\mathcal{Q}_i \subseteq \pi(|HM_i : HM_i \cap CM_i|)$ for all $i$. Then $d_X(G) \leq 2$.
\end{prop}

For the proof of Proposition \ref{in mezzo} we need the following result, which is a consequence of \cite[Theorem 2.2]{Prob}.

\begin{lem}\label{diagX}
Let $L$ be a monolithic group with non-abelian socle $A$, and let $X\subseteq L$. If 
$t\leq|\pi(A)|+1$, then
\[d_{\diag(X)}(L_t) \leq \max\{2,d_{XA}(L)\}.\]
\end{lem}
	
\begin{proof}[Proof of Proposition \ref{in mezzo}]
Let $N \unlhd G$ such that $N$ is maximal with respect to the property that 
\[d_{XN}(G) = d_X(G) = d.\]
By Proposition \ref{G&L&X}, there exists a monolithic group $L$ and an isomorphism 
$\varphi: G/N\rightarrow L_t$ satisfying the conditions of Lemma \ref{lemmabasedivisors}. Moreover, $\varphi(XN/N) = \diag (Y)$ for some $Y\subseteq L$. 
Consider the series
\[N / N \leq M_nN/N\leq\ldots\leq M_iN/N\leq\ldots\leq G/N\cong L_t.\]
Then there exists an index $0\leq h \leq n$ such that $M_hN/N=\Soc(G/N)$ and, by the maximality of $N$, $d_{XM_h}(G) = d_{YA}(L) \leq d-1$.

If $A$ is abelian, then Lemma \ref{lemmabasedivisors} implies that $t=1$, so $G/N$ is isomorphic to $L$ and Lemma \ref{maind(L)Y} yields
\[d_X(G) = d_Y(L) \leq \max\{2, d_{YA}(L)\} \leq \max \{2, d-1\} \le 2.\]
If $A$ is non-abelian, then using Lemma \ref{lemmabasedivisors} we get $t \leq \beta +1$, where 
\[\beta = |\mathcal{Q} \cap \pi(|A|)| \leq |\pi(A)|.\]
By Lemma \ref{diagX} we have
\[d_X(G) = d_{{\rm diag}(X)}(L_t) \leq \max	\{2, d_{YA}(L)\},\]
which implies that $d_X(G) \leq 2$ also in this case.
\end{proof}
	
Using Proposition \ref{in mezzo}, we can can now prove a generalization of \cite[Theorem 1.5]{Prob}, which will be our main tool in the proof of Theorem B. 

\begin{prop}\label{pdfinitoGEN}
Let $G$ be a finite group and let
\[1=M_n \le \cdots \le M_i \le \cdots \le M_0=G\]
be a chief series of $G$. Let $X,C$ be two subgroups of $G$ such that $X\leq C$. Then there exist $a,b\in G$ such that 
\[\pi(|G : CM_i|) \subseteq \pi(|\langle a,b,X \rangle M_i : CM_i \cap \langle a,b,X \rangle M_i|)\]
for any $i$.
\end{prop}

\begin{proof}
As in the proof of Proposition \ref{p.g.gen}, it is enough to show that if $G$ has no proper subgroup $H$ such that $X \le H$ and $\pi(|G : CM_i|) \subseteq \pi(|HM_i : HM_i \cap CM_i|)$ for all $i$, then $d(G)\leq 2$. Applying Proposition \ref{in mezzo}, with $\mathcal{Q}_i = \pi(|G : CM_i|)$, we deduce that $d_X(G) \leq 2$. Since this holds for \emph{any} subgroup $X$ of $G$, it is true for $X=1$. The result follows.  
\end{proof}

We are now ready to prove Theorem B.


\begin{proof}[Proof of Theorem B]
Let $G$ be a profinite group and let $X,C$ be closed subgroups of $G$ such that $X \le C$. By definition, we have 
\[|G : C| = \lcm_{N\unlhd_o G} |G : CN|.\]
In particular, if $p\in \pi(|G : C|)$ then there exists an open normal subgroup $N$ of $G$ such that $p$ divides $|G : CN|$. Let $\mathcal{N}$ be the set of such normal subgroups and set $M = \bigcap_{N\in\mathcal{N}}N$, so $G/M$ is a profinite group. Note that $\mathcal{N}$ is countable. By arguing as in the proof of Theorem A (see the first paragraph), we may assume that $G$ is countably based, so $G$ has a chain 
\begin{equation}\label{e:chain2}
G = M_0 \ge M_1 \ge M_2 \ge \cdots
\end{equation}
of open normal subgroups comprising a base of open neighbourhoods of $1$.
		
Note that $\pi ( |G/M_i : CM_i/M_i| ) = \pi ( |G : CM_i| )$ for any $i$. For each $M_i$ in \eqref{e:chain2}, define
\[\Omega_i = \{ (x_1,x_2) \in G^2 \mid \pi(|G : CM_i|) \subseteq \pi(| \langle x_1,x_2,X \rangle M_i : CM_i \cap \langle x_1,x_2,X \rangle M_i|)\}.\]
Since $G/M_i$ is finite, \cite[Theorem 1.5]{Prob} implies that there exist elements $aM_i,bM_i\in G/M_i$ such that
\[\pi(|G/M_i : CM_i/M_i|) \subseteq \pi(| \langle a,b,X \rangle M_i/M_i : CM_i/M_i \cap \langle a,b,X \rangle M_i/M_i|),\]
where the second set is equal to $\pi(| \langle a,b,X \rangle M_i : CM_i \cap \langle a,b,X \rangle M_i|)$. Hence $\Omega_i$ is non-empty for all $i$. Moreover, if $(x_1,x_2)\in \Omega_i$ then $x_1M_i \times x_2M_i \subseteq \Omega_i$. In fact, $\Omega_i$ is the union of finitely many subsets of this type. Since $M_i$ is closed in $G$, it follows that $x_1M_i \times x_2M_i$ is closed in $G\times G$, hence $\Omega_i$ is closed for all $i$. 

Consider a finite subchain $M_{i_1} \leq \cdots \leq M_{i_r}$
of \eqref{e:chain2}. Then
\[M_{i_1} / M_{i_1} \leq M_{i_2}/M_{i_1} \leq \cdots \leq M_{i_r}/M_{i_1} \leq G/M_{i_1}\]
is a normal series for $G/M_{i_1}$. Since $G/M_{i_1}$ is finite, we can apply Proposition \ref{pdfinitoGEN}, which implies that there exist $x_1,x_2\in G/M_{i_1}$ such that 
\[\pi(|G : CM_{i_1}\,M_{i_j}|) \subseteq \pi(|\langle x_1,x_2,XM_{i_1}\rangle M_{i_j} :  CM_{i_1}\,M_{i_j} \cap \langle x_1,x_2,XM_{i_1} \rangle M_{i_j}|)\]
for all $j=1,\ldots,r$.
			
Since $x_1=xM_{i_1}$, $x_2=yM_{i_1}$ with $x,y\in G$, and $M_{i_1} \subseteq M_{i_j}$ for all $j=1,\ldots,r$, we get
\[\pi(|G : CM_{i_j}|) \subseteq \pi(|\langle x,y,X \rangle M_{i_j} :  CM_{i_j} \cap \langle x,y,X \rangle M_{i_j}|)\]
and thus $(x,y)\in\bigcap_{j=1}^r\Omega_{i_j}$. So every finite subcollection of the $\Omega _i$ has non-empty intersection. Since $G$ is compact, the intersection of all the $\Omega _i$ is non-empty. In other words, there exist $a,b\in G$ such that
\[\pi(|G : CM_i|) \subseteq \pi(|\langle a,b,X \rangle M_i :  CM_i \cap \langle a,b,X \rangle M_i|)\]
for any $M_i$ in the chain \eqref{e:chain2}.
		
Finally, we will show that $a$ and $b$ are the two elements that we are looking for. Let $p$ be a prime in $\pi(|G : C|)$. Then there exists $N\in\mathcal{N}$ such that $p$ divides $|G : CN|$, and there exists a subgroup $M$ in the chain \eqref{e:chain2} with the property that $M \leq N$, so $p\in\pi(|G : CM|)$. Therefore, we have  
\[p\in\pi(|G : CM|) \subseteq \pi(|\langle a,b,X \rangle M : CM \cap \langle a,b,X \rangle M|),\]
which implies $p\in\pi(|\langle a,b,X \rangle : C\, \cap \langle a,b,X\rangle |)$. This concludes the proof of Theorem B.
\end{proof}
	
In the statement of Theorem B, if we take $X$ to be the trivial subgroup, or if we take $X=C$, then we obtain the following corollaries:

\begin{cor}\label{X=1}
Let $G$ be a profinite group and let $C$ be a closed subgroup of $G$. Then there exist $a,b\in G$ such that
\[\pi(|G : C|) \subseteq \pi(|\langle a,b \rangle  : C \cap \langle a,b \rangle|).\]
\end{cor}
	
\begin{cor}\label{X=C}
Let $G$ be a profinite group and let $C$ be a closed subgroup of $G$. Then there exist $a,b\in G$ such that
\[\pi(|G : C|) = \pi(|\langle a,b,C \rangle : C|).\]
\end{cor}
	
Using Theorem B, we can also deduce another important result. Let $G$ be a finite group and set ${\rm Ind}_{G}(x) = |G:C_G(x)|$ for $x \in G$. In \cite{Cam_sica}, Camina, Shumyatsky and Sica prove that if ${\rm Ind}_{\la a,b,x\ra}(x)$ is a prime-power for any $a,b\in G$, then ${\rm Ind}_G(x)$ is also a prime-power. Equivalently, if $C=C_G(x)$ and there is more than one prime dividing $|G:C|$, then there exist $a,b\in G$ such that $|\langle a,b,x \rangle : C \cap \langle a,b,x\rangle|$ is divisible by more than one prime. The following theorem provides a profinite analogue of this result; it is an immediate corollary of Theorem B.

\begin{thm}\label{t:css}
Let $G$ be a profinite group and let $X,C$ be subgroups of $G$ such that $\ds{X\leq C}$. If $|\pi(|G : C|)|\geq 2$, then there exist $a,b \in G$ such that $|\pi(|\langle a,b, X \rangle : C\cap\langle a,b,X \rangle |)| \geq 2.$
\end{thm}

\section{Proof of Theorem C}\label{s:c}

In this final section we prove Theorem C. Let $G$ be a profinite group and recall that the \textit{exponent} of $G$, denoted $\exp(G)$, is defined by
\[\exp(G)=\lcm\{|g| \mid g\in G\},\]
where $|g|:=|\ol{\la g \ra}|$ is the profinite order of $g$. Note that $|g|=\prod_p p^{n(p,g)}$ is a supernatural number, and thus $\exp(G)$ is also a supernatural number. Indeed 
\[\exp(G)=\lcm_{g\in G}\prod_p p^{n(p,g)}=\prod_p p^{n(p)}\]
where $\displaystyle{n(p)=\operatorname*{max}_{g\in G} n(p,g)}$.
	
We are interested in the following general problem.

\begin{prob}\label{p:p1}
Is there a constant $d$ such that every profinite group $G$ contains a $d$-generated (closed) subgroup $H$ such that $\exp(H)=\exp(G)$?
\end{prob}
	
For finite groups, this question has a positive answer with $d \le 3$ (see \cite[Theorem 1.6]{Prob}), but it is not known whether or not this bound is sharp (by \cite[Theorem E]{Bound}, $d=2$ if $G$ is solvable). 

Recall that a group is a \emph{torsion} group if every element has finite order. The following closely related problem is still open:
	
\begin{prob}\label{p:p2}
Does every torsion profinite group have finite exponent? 
\end{prob}
		
A partial solution to this problem was given by Zel'manov in 1992 (\cite[Theorem 1]{Zel}):
	 
\begin{thm}[Zel'manov]\label{Zel'manov}
Every finitely generated pro-$p$ torsion group is finite.
\end{thm}
	
We claim that if the solution to Problem \ref{p:p1} is positive, then we can use Zel'manov's theorem to solve Problem \ref{p:p2}. To see this, let us first recall the following result (see 
\cite[Theorem 1]{Herf}):
	
\begin{thm}
Let $G$ be a torsion profinite group. Then the order of $G$ is divisible by only finitely many distinct primes.
\end{thm}
	
Therefore, if $G$ is torsion, then $|G|$ has a finite number of prime divisors, say $p_1,\ldots,p_t$. Let $P_1,\ldots,P_t$ be the corresponding Sylow subgroups of $G$ (for a prime $p$, a $p$-Sylow subgroup $P$ of $G$ is a subgroup $P$ whose order is a power of $p$ (possibly $p^{\infty}$) and its index is not divisible by $p$).  If Problem \ref{p:p1} has a positive solution, then for any $i$ there exists a finitely generated subgroup $Q_i \subseteq P_i$ such that $\exp(Q_i) = \exp(P_i)$, and $Q_i$ is finite by Zel'manov's theorem. This implies that $\exp(G)$ is finite. 

Making progress on Problem \ref{p:p1} is fairly difficult, and a complete solution is still out of reach. Nevertheless, if $P$ is a finitely generated pro-$p$ group then either $P$ is torsion (and thus finite by Zel'manov's theorem), or it is infinite and non-torsion. In the former case, $\exp(P)$ is finite and there exists an element $x \in P$ such that $\exp(P)=|x|$. In the latter case, since $P$ is non-torsion, there is an element of order $p^{\infty}$, which is exactly the exponent of $P$ because $\exp(P)=\lcm\{|y| \mid y\in P\}$ by definition.
	 
Using these observations, it is possible to prove Theorem C, which gives a best possible solution to Problem \ref{p:p1} in the case where $G$ is a finitely generated \textit{prosupersolvable} group, which is an inverse limit of finite supersolvable groups. In order to prove this result, we require some additional terminology.

Let $\pi$ be a set of prime numbers and let $\pi'$ be the set of prime numbers not in $\pi$. We say that a supernatural number $\delta$ is a \emph{$\pi$-number} if the primes dividing $\delta$ are in $\pi$. A closed subgroup $H$ of a profinite group $G$ is a \emph{$\pi$-subgroup} if the order of $H$ is a $\pi$-number. If $H$ is a maximal $\pi$-subgroup of $G$, it is called a \emph{$\pi$-Sylow subgroup} of $G$, and it is called a \emph{$\pi$-Hall subgroup} if it is a $\pi$-Sylow subgroup and $|G : H|$ is a $\pi'$-number.
	 
\begin{remk}
If $G$ is prosolvable, every $\pi$-Sylow subgroup is a $\pi$-Hall subgroup, and any two of them are conjugate (see \cite[Corollary 2.3.7]{Ribes}). In particular, this property holds for prosupersolvable groups.
\end{remk}

\begin{thm}\label{split}
Let $n$ be a positive integer and let $\pi$ be the set of prime numbers greater than $n$. Let $G$ be a prosupersolvable group. 
\begin{itemize}
\item[{\rm (i)}] There is a (unique) normal $\pi$-Sylow subgroup $K$ of $G$.
\item[{\rm (ii)}] There is a split exact sequence of prosupersolvable groups 
\begin{center}
\begin{tikzpicture}
\matrix (m) [matrix of math nodes,column sep=2em,minimum width=1em]
{  1 & K & G & H & 1 \\};
 \path[->,<-]
  (m-1-1) edge[->] node [] {} (m-1-2)
  (m-1-2) edge[->] node [] {} (m-1-3) 
  (m-1-3)  edge[->] node  [below] {$\varphi$} (m-1-4)
  (m-1-4) edge[->] node [] {} (m-1-5);
\end{tikzpicture}
\end{center}
where $\varphi$ is an open map and $H$ is a $\pi'$-Sylow subgroup of $G$. In other words, $G$ is the topological semidirect product of $K$ and $H$.
\end{itemize}
\end{thm}

\begin{proof}
This is \cite[Proposition 3.5]{pss}.
\end{proof}
	 
Before embarking on the proof of Theorem C, we record some additional notation and  preliminary results on prosupersolvable groups.
	 	
Recall that the \textit{Frattini subgroup} of a profinite group $G$, denoted $\Phi(G)$, is the intersection of all the maximal closed subgroups of $G$. The following result is \cite[Corollary 3.9]{pss}.
		
\begin{thm}\label{pro-p f.g.}
Let $G$ be a finitely generated prosupersolvable group. For each prime number $p$, let $P$ be a $p$-Sylow subgroup of $G$. Then $P$ is finitely generated.
\end{thm}
	
We will also need also Gasch\"utz's Theorem (see \cite{Gas}), and its profinite version (\cite[Lemma 2.1]{Gash}):
	
\begin{thm}[Gasch\"utz]\label{Gas} 
Let $N$ be a normal subgroup of a finite group $G$ and let $g_1,\ldots,g_m \in G$ be such that $G=\langle g_1,\ldots,g_m,N \rangle$. If $d(G)\leq m$, then there exist elements $u_1,\ldots,u_m$ of $N$ such that $G = \langle g_1u_1,\ldots,g_mu_m \rangle$.
\end{thm}
	
\begin{thm}\label{Gas_pro}
Let $\pi : G \rightarrow H$ be a continuous epimorphism from a finitely generated profinite group onto $H$. Assume $d(G) \leq d$ and write $H = \la z_1,\ldots,z_d \ra$.  Then there exist $y_1,\ldots,y_d \in G$ that generate $G$ and $\pi(y_i)=z_i$ for every $i=1,\ldots,d$.
\end{thm}
	
In the proof of Theorem C, one of the key ideas is to inductively construct a chain of closed subsets of $G \times G$. In order to do this, the following two results will play a fundamental role.
	
\begin{prop}\label{2-gen}
Let $G = \langle g \rangle ^H\rtimes H$ be a finite solvable group, where $H$ is a $2$-generated subgroup, $g$ is a $p$-element where $p$ is a prime, $p\nmid |H|$ and $\langle g \rangle ^H$ is a $p$-group. Then $G$ is $2$-generated.
\end{prop}
	
\begin{proof}
Consider $F:= \Phi(\langle g \rangle ^H)$. Since $F\subseteq \langle g \rangle^H$ we have 
\[X=G/F \cong M\rtimes H\]
where $M$ is an elementary abelian $p$-group, which can be viewed as an $\mathbb{F}_p[H]$-module. Since $p\nmid |H|$, Maschke's Theorem implies that $M$ is a completely reducible 
$\mathbb{F}_p[H]$-module, so we can write 
\[M=\prod_i N_i^{k_i}\]
where, for any $i$, $N_i$ is an irreducible $\mathbb{F}_p[H]$-module and $N_i\ncong N_j$ are non-isomorphic $\mathbb{F}_p[H]$-modules, for all $i\neq j$.

\vspace{2mm}
		
\noindent \textbf{Claim.}  $X$ is 2-generated. 

 
\begin{proof}[Proof of claim]
Let $N\unlhd X$ such that $d(X/N) = d = d(X)$ and any proper quotient of $X/N$ can be generated with $d-1$ elements. By Theorem \ref{G&L},  $X/N \cong L_k$ for some integer $k$ and some monolithic group $L$ whose socle $A=\Soc(L)$ is complemented and abelian, since $G$ is solvable. 
		
Seeking a contradiction, suppose that $d(X) > 2$. Then $d(L_k) \geq 3$ and $k\geq 2$. Indeed,  if $k=1$ then Proposition \ref{Maind(L_k)}(i) implies that $d = d(L) = \max \{ 2, d(L/A) \} = d(L/A)$, but this contradicts the fact that every proper quotient of $X/N\cong L_k$ can be generated with $d-1$ elements. 
		
If $A$ is not a $p$-group, then $L_k$ is an epimorphic image of $H$, which is $2$-generated, but this contradicts the fact that $L_k$ needs at least $3$ generators. Therefore, $A$ is a $p$-group. In particular, the minimality of $A$ implies that it is an irreducible $\mathbb{F}_p[H]$-module, and $A^k$ is isomorphic (as an $\mathbb{F}_p[H]$-module) to an epimorphic image of $M$. Since $M$ is a cyclic $\mathbb{F}_p[H]$-module, it follows that $A^k$ is also cyclic.
		 
Write $|{\rm End}_{L/A}(A)|=q$ and $|A|=q^r$, with $q$ a $p$-power. Let $J$ be the Jacobson radical of $\mathbb{F}_p[H]$, so $\mathbb{F}_p[H] / J$ is a semisimple algebra. By general properties of the Jacobson radical, $A^k$ is also a cyclic $\mathbb{F}_p[H] / J$-module.  If $A$ occurs $n$ times in $\mathbb{F}_p[H] / J$, then by applying \cite[Lemma 1]{cos}, with 
$\Lambda=\mathbb{F}_p[H] / J $ and $W=A^k$, it follows that $\lceil k / n \rceil=1$ and thus $k \leq n$. Since $\mathbb{F}_p[H] / J$ is a semisimple algebra and also Artinian (since it is finite), we can use the Wedderburn-Artin theorem (see Lemma 1.11, and Theorems 1.14 and 3.3 in \cite{Hung}) to deduce that
\[n = \dim_{\End_H(A)}(A) = \dim_{\End_{L/A}(A)}(A) = r.\]
Therefore, $k \leq r$.
		
Since $X$ is solvable and $(|L/A| , |A| ) = 1$, the Schur-Zassenhaus theorem implies that all complements of $A$ in $L$ are conjugate, so $H^1(L/A,A)=\{0\}$  (see \cite{Coho}), which implies that $|H^1(L/A,A)|=q^s=1$ and thus $s=0$. Moreover, $L/A \cong L_k/A^k$ is a proper quotient of $L_k\cong X/N$, hence $d(L/A)\leq d-1$.
		
Finally, by Proposition \ref{Maind(L_k)}(ii) we get
\[d = \textrm{d}(X/N) = d(L_k) = \max\{d-1 ,\theta + \lceil k/r \rceil\} \leq 2\]
hence $X$ is $2$-generated. This justifies the claim.
\end{proof}

\vspace{2mm}
		
In conclusion, since $X$ is the quotient of $G$ by the Frattini subgroup of $\langle g \rangle ^H$, it follows that $d(G)=d(X)$. Therefore, $G$ is also $2$-generated, and this completes the proof of the proposition.
\end{proof}
		
We can extend the previous proposition to profinite groups.
		 
\begin{thm}\label{2-gen_pro}
Let $G = \langle g \rangle ^H\rtimes H$ be a prosolvable group, where $H$ is a $2$-generated subgroup, $g$ is a $p$-element with $p$ a prime such that $p\nmid |H|$, and 
$\langle g \rangle ^H$ is a pro-$p$ group. Then $G$ is $2$-generated.
\end{thm}

\begin{proof}
By \cite[Proposition 4.2.1]{Wilson}, a profinite group $X$ is $2$-generated if and only if every finite continuous image of $X$ is $2$-generated. In particular, $X$ is $2$-generated if and only if $X/N$ is $2$-generated for any open normal subgroup $N$ of $X$. 
		
Let $N$ be an open normal subgroup of $G$ and set $P:=\langle g \rangle ^H$. Consider 
$N\cap P$. Since $N\cap P$ is normal in $G$, we can consider the quotient $G^\ast =G/(N\cap P)=P^\ast \rtimes H$, where $P^\ast\cong P/(N\cap P)$ is finite. Let $N^\ast$ be the image of $N$ in $G^\ast$. Note that $N^\ast \cap P^\ast=1$, hence $p\nmid |N^\ast|$ and $N^\ast$ is a $p'$-subgroup of $G^\ast$. Since $H$ is a $p'$-Hall subgroup of $G^\ast$, the profinite version of the Schur-Zassenhaus theorem (see Propositions 2.3.2 and 2.3.3 in \cite{Wilson}) implies that  there exists an element $x\in G^\ast$ such that $x^{-1} N^\ast x \leq H$, hence $N^\ast \leq H$ since $N^\ast$ is normal in $G^\ast$. Moreover, $[N^\ast , P^\ast]\subseteq N^\ast \cap P^\ast = 1$, hence $N^\ast \subseteq C_H(P^\ast)$, where $C_H(P^\ast)$ is the centralizer of $P^\ast$ in $H$. This implies that $\langle g \rangle ^H=\langle g \rangle ^{H/N^\ast}$, hence $G^\ast / N^\ast = P^\ast \rtimes H/N^\ast$ is a finite solvable group since $H/N^\ast$ is finite. 
		
By applying Proposition \ref{2-gen}, we deduce that $G^\ast/N^\ast$ is $2$-generated hence, by our initial observation on finite continuous images, we deduce that $G$ is $2$-generated.
\end{proof}
		
We are now ready to prove Theorem C. 

\begin{proof}[Proof of Theorem C]
The main idea is to construct a chain of closed subsets of $G \times G$ such that the intersection of these subsets contains a pair of elements generating a subgroup $C$ with the required property.
		
To get started, let us construct a countable chain $G = M_0 \ge M_1 \ge M_2 \ge \cdots$ of normal subgroups of $G$, with the property that $\bigcap_i M_i = 1$ and the quotient $M_i / M_{i+1}$ is a Sylow subgroup of $G/M_{i+1}$ for each $i$. 

Let $p_0$ be the smallest prime which divides the order of $G$, and let $\pi_1$ be the set of prime numbers greater than $p_0$. By Theorem \ref{split}, there exists a normal $\pi_1$-Sylow subgroup $M_1$ of $G$ such that $K_0 = G / M_1$ is a $p_0$-Sylow group, and $G$ is a semidirect product $G = M_1 \rtimes K_0$. Now we turn to $M_1$. Let $p_1$ be the smallest prime which divides the order of $M_1$, and let $\pi_2$ be the set of primes greater than $p_1$. Again, using Theorem \ref{split}, there exists a normal $\pi_2$-Sylow subgroup $M_2$ of $M_1$ such that $K_1 = M_1 / M_2$ is a $p_1$-Sylow group, and $M_1$ is a semidirect product $M_1 = M_2 \rtimes K_1$. Now we consider $M_2$ and repeat the procedure, obtaining $M_3$ and so on. In this way, we obtain the following chain
\[G = M_0 \ge M_1 \ge M_2 \ge \cdots \ge M_i \ge M_{i+1} \ge \cdots\]
where for any $i$, $M_i/M_{i+1}$ is isomorphic to a $p_i$-Sylow subgroup of $G$, and neither $|G/M_i|$ nor $|M_{i+1}|$ is divisible by $p_i$.
		
Next we construct a chain of closed subsets $\Omega_i$ of $G\times G$ with the following properties:
\begin{enumerate}
\item[1.] $\Omega_i = x_iM_i \times y_iM_i$;
\item[2.] $x_i \equiv x_{i-1}\textrm{ mod }M_{i-1}$ and $y_i \equiv y_{i-1}\textrm{ mod }M_{i-1}$;
\item[3.] $\exp(G/M_i) = \exp( \langle x_i,y_i \rangle M_i/M_i)$.
\end{enumerate}

Note that the second property implies that $\Omega_{i} \subset \Omega_{i-1}$ for each $i$. This chain can be constructed via induction on the index $i$. 

To see this, suppose that $H_i = \langle x_i,y_i \rangle M_i/M_i$ is a $2$-generated closed subgroup with the same exponent as $G/M_i$. The factor group $M_i / M_{i+1}$ is isomorphic to $P$, a $p_i$-Sylow subgroup of $G$, and Theorem \ref{pro-p f.g.} implies that $P$ is finitely generated. As we have previously observed, Zel'manov's theorem (Theorem \ref{Zel'manov}) implies that $P$ contains an element $g$ whose order is equal to the exponent of $P$.
		 
Consider $H_{i+1} = \langle g \rangle ^{H_i} \rtimes H_i$. By Propositions \ref{2-gen} and \ref{2-gen_pro}, $H_{i+1}$ is a $2$-generated closed subgroup of $G/M_{i+1}$ and we can write 
$H_{i+1} = \langle x_{i+1},y_{i+1} \rangle M_{i+1}/M_{i+1}$ by Gasch\"utz's theorem, where $x_{i+1}=x_ia_i$ and $y_{i+1}=y_ib_{i}$ with $a_i,b_i\in M_i$. Moreover, $\exp(H_{i+1}) = \exp(
G/M_{i+1})$. We can now set  
\[\Omega_{i+1} = x_{i+1}M_{i+1} \times y_{i+1}M_{i+1},\]
which is closed in $G \times G$ since $M_{i+1}$ is closed in $G$. Note that $x_{i+1}M_{i+1} = x_ia_iM_{i+1} \subset x_iM_i$, and similarly $y_{i+1}M_{i+1} \subset y_iM_i$. Hence, we get the following chain of closed subsets of $G \times G$:
\[\Omega_1 \supset \Omega_2 \supset \cdots \supset \Omega_i \supset \Omega_{i+1} \supset \cdots\]
Note that every finite subchain has non-empty intersection, so the compactness of $G$ implies that $\bigcap_{i \ge 0}\Omega _i$ is non-empty.
	 
Let $(c_1,c_2)\in\bigcap_i\Omega_i$ and consider $C=\overline{ \langle c_1,c_2 \rangle }$. We will show that this is the $2$-generated closed subgroup of $G$ that we are looking for.
	 
Trivially $\exp(C) \leq \exp(G)$. Conversely, let $p^{n(p)}$ be the largest power of $p$ dividing $\exp(G)$, where $n(p)\leq \infty$. As previously noted, a $p$-Sylow subgroup $P$ of $G$ contains an element  whose profinite order is equal to the exponent of $P$. Moreover, there exists $i$ such that $P \cong M_{i-1}/M_i$ and $\exp(P)=\exp(M_{i-1}/M_i)$. Since $(c_1,c_2) \in \Omega_i = x_iM_i \times y_iM_i$, it follows that 
\[\exp(G/M_i)  = \exp(\langle x_i, y_i \rangle M_i /M_i) = \exp(\langle c_1, c_2 \rangle M_i /M_i) 
 = \exp(CM_i / M_i) = \exp(C / C \cap M_i)\]
divides $\exp(C)$. Since $\exp(M_{i-1}/M_i)$ divides $\exp(G/M_i) = \exp(\langle x_i,y_i \rangle M_i/M_i)$, we can conclude that $\exp(P) = p^{n(p)}$ divides $\exp(C)$. As this holds for every prime $p$, we deduce that $\exp(C) = \exp(G)$.
	
This completes the proof of Theorem C.
\end{proof}

\vspace{2mm}

\noindent \textbf{Acknowledgments.} The author would like to thank her Master's degree supervisor, Professor Andrea Lucchini, who helped her to think up the results presented above, and her Ph.D. supervisor, Dr. Tim Burness, for his encouragement and precious comments on  earlier drafts of the present paper.


\end{document}